\documentclass[11pt]{amsart}

\usepackage{amscd}
\usepackage{amsmath, amssymb}
\usepackage{amsfonts}
\usepackage{enumerate}
\newcommand{\vp}{\varphi}
\newcommand{\ve}{\varepsilon}
\newcommand{\ddbar}{\sqrt{-1} \partial \overline{\partial}}

\begin{document}
\newcounter{remark}
\newcounter{theor}
\setcounter{remark}{0}
\setcounter{theor}{1}
\newtheorem{claim}{Claim}
\newtheorem{theorem}{Theorem}[section]
\newtheorem{lemma}[theorem]{Lemma}
\newtheorem{corollary}[theorem]{Corollary}
\newtheorem{proposition}[theorem]{Proposition}
\newtheorem{question}{question}[section]
\newtheorem{defn}{Definition}[theor]
\numberwithin{equation}{section}

\title[Optimal regularity of plurisubharmonic envelopes]{Optimal regularity of plurisubharmonic envelopes on compact Hermitian manifolds}
\author{Jianchun Chu}
\address{School of Mathematical Sciences, Peking University, Yiheyuan Road 5, Beijing, P.R.China, 100871}
\email{chujianchun@pku.edu.cn}

\author[Bin Zhou]{Bin Zhou*}
\address{School of Mathematical Sciences, Peking University, Yiheyuan Road 5, Beijing, P.R.China, 100871}
\email{bzhou@pku.edu.cn}


\thanks {*Partially supported by NSFC 11571018 and 11331001}

\begin{abstract}
In this paper, we prove the $C^{1, 1}$-regularity of the plurisubharmonic envelope of a $C^{1,1}$ function on a compact Hermitian manifold. We also present examples to show this regularity is sharp.
\end{abstract}
\maketitle

\section{Introduction}
The subharmonic envelope is an important tool in the classical potential theory for Laplacian equation.
This notion can be extended to the potential theory of nonlinear elliptic equations, and  the issue of regularity of the envelope also arises naturally.
For convex envelopes, the optimal $C^{1,1}$-regularity has been
confirmed recently in \cite{DF15}. For complex Monge-Amp\`ere equations on a domain in $\mathbb C^n$,
the Perron-Bremermann plurisubharmonic upper envelope has been studied in \cite{BT}.
It is also interesting to establish the regularity for plurisubharmonic envelopes on  complex manifolds.
On a K\"ahler manifold, the plurisubharmonic envelopes have been studied for cohomology classes of great extent, including big classes \cite{BD, BEGZ10}.

Let $(M,\omega)$ be a compact Hemitian manifold of complex dimension $n$
and $PSH(M,\omega)$ be the set of $\omega$-plurisubharmonic functions \cite{S}.
For any function $f$ on $M$,  following \cite{BD, BEGZ10}, we define its {\it plurisubharmonic envelope (or extremal function)} by
\begin{equation}\label{Definition of envelope}
\vp_{f}(x)=\sup\{\vp(x)\ |\ \vp\in PSH(M,\omega) \text{~and~} \vp\leq f\},\ x\in M.
\end{equation}
Then $\vp_{f}\in PSH(M,\omega)$. Moreover,
it is shown in \cite{BD} that, when $\omega$ is K\"ahler and $f\in C^{\infty}(M)$, $\vp_f\in C^{1,\alpha}(M)$ for any $\alpha\in(0,1)$.
It is expected that the optimal regularity for this envelope is $C^{1,1}$, which has been realized when
$[\omega]$ is an integral class \cite{Be09, RN}. We prove the sharp regularity for general Hermitian manifold in this paper.

The idea of the proof is to consider the envelope as the solution to an obstacle problem for the  complex Monge-Amp\`ere equation \cite{Be}.
The similar treatment for the real Monge-Amp\`ere equations and other equations can be found in \cite{Lee, Ob,  DWZ}. Then the regularity relies
on the a priori estimates of the solutions to the following complex Monge-Amp\`ere equations
\begin{equation}\label{CMA0}
(\omega+\ddbar\vp)^{n}=e^{\frac{1}{\ve}(\vp-f)}\omega^{n}
\end{equation}
for small $\ve>0$.
It is well-known that in K\"ahler case,  the solution to the above equation
 has been established by \cite {Aub76,Yau78}. The solvability has been extended to the Hermitian case by \cite{C, GuLi}.
Let $\vp_\ve$ be the solution to (\ref{CMA0}).
Then we have

\begin{theorem}\label{C 1,1 regularity of envelope}
Let $(M,\omega)$ be a compact Hermitian manifold and $f\in C^{1,1}(M)$. Then we have $\vp_\ve$ converges to $\vp_{f}$ and there is a constant $C$ independent of $\ve$ such that $\|\vp_\ve\|_{C^{2}(M)}\leq C$. In particular, $\vp_f\in C^{1,1}(M)$.
\end{theorem}

It would be also interesting to study the regularity of envelopes with prescribed singularity as
in \cite{Be09, RN}. However, there are still difficulties in deriving the a priori estimates.


The paper is organized as follows.
In Section 2, we establish the uniform a priori estimates for the Monge-Amp\`ere equation \eqref{CMA0}.
In particular, we apply the new techniques in \cite{ChToWe16a} with a modification of the
auxilary function to estimate the second order derivatives. Theorem \ref{C 1,1 regularity of envelope} is proved in Section 3.
In the last section, we give some examples showing that the $C^{1,1}$-regularity is optimal.

\bigskip

{\bf Acknowledgments.} The first-named author would like to thank his advisor G. Tian for encouragement and support. After finishing writing this preprint, we learned that Theorem 1.1 in the case of K\"{a}hler manifolds is independently obtained by Tosatti \cite{To} and solved a problem of Berman.

\section{The a priori estimate}
Let $(M,\omega)$ be a compact Hermitian manifold of complex dimension $n$. We use $g$ and $\nabla$ to denote the corresponding Riemannian metric and Levi-Civita connection (Note that we use Levi-Civita connection, not Chern connection). In this section, we study the a priori estimates of the following complex Monge-Amp\`{e}re equation
\begin{equation}\label{CMAE}
\left\{ \begin{array}{ll}
(\omega+\sqrt{-1}\partial\bar{\partial}\vp)^{n}=e^{\frac{1}{\ve}(\vp-F)}\omega^{n}\\[4pt]
\omega+\sqrt{-1}\partial\bar{\partial}\vp>0
\end{array},\right.
\end{equation}
where $\ve\in(0,1)$ is a constant and
 $F$ is a real-valued $C^{2}$ function on $M$.
The solvability of the equation can be guaranteed by \cite {C, GuLi}.
For our purpose, stronger estimates are needed.
We write $\tilde{\omega}=\omega+\ddbar\vp$  and $\tilde{g}$ be the corresponding Riemannian metric
for convenience. We often use $C$ to denote a uniform constant depending only on $\|F\|_{C^{2}}$ and $(M,\omega)$. All norms $\|\cdot\|_{C^{k}}$ in this paper are taken with respect to $(M,\omega)$. And all the following estimates are uniform with respect to $\ve$.

\begin{proposition}\label{Zero order estimate}
Let $\varphi$ be a smooth solution to (\ref{CMAE}). Then we have
\begin{equation*}
\max_{M}(\vp-F)\leq C_{0}\ve \text{~and~} \min_{M}\vp\geq\min_{M}F,
\end{equation*}
where $C_{0}$ is a constant depending only on $\|F\|_{C^{2}}$ and $(M,\omega)$.
\end{proposition}

\begin{proof}
First, we assume that $(\vp-F)$ attains its maximum at $p\in M$. By maximum principle, it is clear that
\begin{equation*}
\ddbar\vp(p)-\ddbar F(p)\leq 0,
\end{equation*}
which implies
\begin{equation*}
(\vp-F)(p)=\ve\log\frac{(\omega+\ddbar\vp)^{n}}{\omega^{n}}(p)
\leq\ve\log\frac{(\omega+\ddbar F)^{n}}{\omega^{n}}(p)\leq C_{0}\ve.
\end{equation*}
By the definition of $p$, we obtain
\begin{equation}\label{Zero order estimate equation 1}
\max_{M}(\vp-F)\leq C_{0}\ve.
\end{equation}
Next, we assume that $\vp(q)=\displaystyle\min_{M}\vp$ for $q\in M$. By a similar argument, we have
\begin{equation*}
(\vp-F)(q)=\ve\log\frac{(\omega+\ddbar\vp)^{n}}{\omega^{n}}(q)\geq 0,
\end{equation*}
which implies
\begin{equation}\label{Zero order estimate equation 2}
\min_{M}\vp=\vp(q)\geq F(q)\geq\min_{M}F.
\end{equation}
Combining (\ref{Zero order estimate equation 1}) and (\ref{Zero order estimate equation 2}), we complete the proof.
\end{proof}

The following proposition is the gradient estimate of (\ref{CMAE}). It is established by Blocki \cite{Blo09} in K\"{a}hler case. For the Hermitian case, we use some calculations in \cite{ChToWe16a} to prove Proposition \ref{Gradient estimate}, but the idea is similar with \cite{Blo09}.

\begin{proposition}\label{Gradient estimate}
If $\vp$ is a smooth solution to (\ref{CMAE}), then there exists a constant $C$ depending only on $\|F\|_{C^{1}}$ and $(M,\omega)$ such that
\begin{equation*}
\sup_{M}|\partial\vp|_{g}\leq C.
\end{equation*}
\end{proposition}

\begin{proof}
First, without loss of generality, we assume $\displaystyle\sup_{M}\vp\leq 0$. Otherwise, we consider the following functions
\begin{equation*}
\psi=\vp-\|\vp\|_{L^{\infty}} \text{~and~} \tilde{F}=F-\|\vp\|_{L^{\infty}}.
\end{equation*}
It then follows that $\displaystyle\sup_{M}\psi\leq 0$  and
\begin{equation*}
 (\omega+\sqrt{-1}\partial\bar{\partial}\psi)^{n}=e^{\frac{1}{\ve}(\psi-\tilde{F})}\omega^{n}.
\end{equation*}
By the definition of $\tilde{F}$ and Proposition \ref{Zero order estimate}, it is clear that $\|\tilde{F}\|_{C^{1}}\leq\|F\|_{C^{1}}+C$.

As in \cite{ChToWe16a} (see Proposition 4.1 in \cite{ChToWe16a}), we consider the following quantity
\begin{equation*}
Q=e^{f(\vp)}|\partial\vp|_{g}^{2},
\end{equation*}
where $f(t)=\frac{1}{A}e^{-A(t-1)}$ and $A$ is a constant to be determined. We assume that $Q$ attains its maximum at $p\in M$. Let $\{e_{i}\}_{i=1}^{n}$ be a local holomorphic frame for $T^{(1,0)}M$ near $p$, such that $\{e_{i}\}_{i=1}^{n}$ is unitary with respect to $g$ and $\tilde{g}_{i\overline{j}}(p)$ is diagonal. For convenience, we write $\vp_{i}=e_{i}(\vp)$ and $\vp_{\overline{i}}=\overline{e}_{i}(\vp)$. By (4.13) in \cite{ChToWe16a} (To avoid confusion of notations, we replace $\ve$ in (4.13) by $\delta$ and $F$ in (4.13) should be replaced by $\frac{1}{\ve}(\vp-F)$), at $p$, for any $\delta\in(0,\frac{1}{2}]$, we have
\begin{eqnarray}\label{Gradient estimate equation 1}
0 &\geq& e^{f}(f''-3\delta(f')^{2})|\partial\vp|_{g}^{2}|\partial\vp|_{\tilde{g}}^{2}
+e^{f}(-f'-C_{0}\delta^{-1})|\partial\vp|_{g}^{2}\sum_{i}\tilde{g}^{i\overline{i}}\nonumber\\
&&+2e^{f}\mathrm{Re}\left(\sum_{i}\frac{1}{\ve}(\vp_{i}-F_{i})\vp_{i}\right)
+(2+n)e^{f}f'|\partial\vp|_{g}^{2}-2e^{f}f'|\partial\vp|_{\tilde{g}}^{2},
\end{eqnarray}
where $C_{0}$ is a constant depending only on $\|F\|_{C^{1}}$ and $(M,\omega)$. Now, we choose $A=12C_{0}$ and $\delta=\frac{A}{6}e^{A(\vp(p)-1)}$. By direct calculations and Proposition \ref{Zero order estimate}, it is clear that
\begin{equation}\label{Gradient estimate equation 2}
f''-3\delta(f')^{2}\geq C^{-1} \text{~and~} -f'-C_{0}\delta^{-1}\geq C^{-1}.
\end{equation}
Without loss of generality, we assume that $|\partial\vp|_{g}\geq\displaystyle\sup_{M}|\partial F|_{g}$ at $p$, which implies
\begin{equation}\label{Gradient estimate equation 3}
\mathrm{Re}\left(\sum_{i}\frac{1}{\ve}(\vp_{i}-F_{i})\vp_{i}\right)
\geq \frac{1}{\ve}\left(|\partial\vp|_{g}^{2}-|\partial F|_{g}|\partial\vp|_{g}\right)\geq 0.
\end{equation}
Combining (\ref{Gradient estimate equation 1}), (\ref{Gradient estimate equation 2}) and (\ref{Gradient estimate equation 3}), we get
\begin{equation*}
0 \geq C^{-1}|\partial\vp|_{g}^{2}|\partial\vp|_{\tilde{g}}^{2}+C^{-1}|\partial\vp|_{g}^{2}\sum_{i}\tilde{g}^{i\overline{i}}-C|\partial\vp|_{g}^{2}-C
\end{equation*}
at $p$. Then by the similar argument of \cite{ChToWe16a}, we obtain $|\partial\vp|_{g}^{2}(p)\leq C$, which completes the proof of Proposition \ref{Gradient estimate}.
\end{proof}

\begin{proposition}\label{Second order estimate}
If $\vp$ is a smooth solution of (\ref{CMAE}), then there exists a constant $C$ depending only on $\|F\|_{C^{2}}$ and $(M,\omega)$, such that
\begin{equation*}
\sup_{M}|\nabla^{2}\vp|_{g}\leq C,
\end{equation*}
where $\nabla$ is the Levi-Civita connection of $(M,\omega)$.
\end{proposition}

\begin{proof}
First, without loss of generality, we assume that $\displaystyle\sup_{M}\vp\leq 0$ as in the proof of Proposition \ref{Gradient estimate}.

For $\nabla^{2}\vp$, we use $\lambda_{1}(\nabla^{2}\vp)$ to denote its largest eigenvalue. Since $\omega+\ddbar \vp>0$, it is clear that $|\nabla^{2}\vp|_{g}\leq C\lambda_{1}(\nabla^{2}\vp)+C$ for a uniform constant $C$ (see (5.1) in \cite{ChToWe16a}). Then we apply maximum principle to the following quantity
\begin{equation*}
Q=\log\lambda_{1}(\nabla^{2}\vp)+h_{D}(|\partial \vp|_{g}^{2})+e^{-A\vp},
\end{equation*}
where
$$h_{D}(s)=-\frac{1}{2}\log(D+\sup_{M}|\partial \vp|_{g}^{2}-s)$$ and $A, D>1$ are constants to be determined. It then follows that
\begin{equation}\label{Properties of h D}
\frac{1}{2D}\geq h_{D}' \geq \frac{1}{2D+2\displaystyle\sup_{M}|\partial \vp|_{g}^{2}} \text{~and~} h_{D}''=2(h_{D}')^{2}.
\end{equation}
Here the definition of $h_{D}$ is different from the definition of $h$ in \cite{ChToWe16a}. In fact, we will choose $D$ suitably to deal with the bad terms arise from the right hand side of (\ref{CMAE}).

We assume the set $\{x\in M~|~\lambda_{1}(\nabla^{2}\vp)>0\}$ is nonempty, otherwise we get $Q\leq C$ directly. Let $p$ be the maximum point of $Q$, i.e., $Q(p)=\displaystyle\max_{M}Q$. As before, we can find local holomorphic frame $\{e_{i}\}_{i=1}^{n}$ for $T^{(1,0)}M$ near $p$, such that
\begin{enumerate}[(1)]
    \item $\{e_{i}\}_{i=1}^{n}$ is unitary with respect to $g$.
    \item At $p$, we have $\tilde{g}_{i\overline{j}}$ is diagonal and
          \begin{equation*}
          \tilde{g}_{1\overline{1}}\geq\tilde{g}_{2\overline{2}}\geq\cdots\geq\tilde{g}_{n\overline{n}}.
          \end{equation*}
\end{enumerate}
Since $(M,\omega)$ is a Hermitian manifold, there exists a real coordinates $\{x^{\alpha}\}_{\alpha=1}^{2n}$ near $p$, such that
\begin{enumerate}[(1)]
    \item At $p$, for any $k=1,2,\cdots,n$, we have
          \begin{equation*}
          e_{k}=\frac{1}{\sqrt{2}}\left(\frac{\partial}{\partial x^{2k-1}}-\sqrt{-1}\frac{\partial}{\partial x^{2k}}\right).
          \end{equation*}
    \item At $p$, for any $\alpha,\beta,\gamma=1,2,\cdots,2n$, we have
          \begin{equation*}
          \frac{\partial g_{\alpha\beta}}{\partial x^{\gamma}}=0.
          \end{equation*}
\end{enumerate}

Let $V_{1}, V_{2}, \cdots, V_{2n}$ be the unit eigenvectors of $\nabla^{2}\vp$ (with respect to $g$) at $p$, corresponding to eigenvalues $\lambda_{1}(\nabla^{2}\vp)\geq\lambda_{2}(\nabla^{2}\vp)\geq\cdots\geq\lambda_{2n}(\nabla^{2}\vp)$. Extend $\{V_{\alpha}\}_{\alpha=1}^{2n}$ to be vector fields near $p$ by taking the components (in above local real coordinates) to be constant.

However, at $p$, $Q$ is not smooth if $\lambda_{1}(\nabla^{2}\vp)=\lambda_{2}(\nabla^{2}\vp)$. To avoid this, we use a perturbation argument (see \cite{Ch,ChToWe16a,ChToWe16b,Sze15,SzToWe15}). As in \cite{ChToWe16a}, near $p$, we consider the following perturbed smooth quantity
\begin{equation*}
\hat{Q}=\log\lambda_{1}(\Phi)+h_{D}(|\partial \vp|_{g}^{2})+e^{-A\vp},
\end{equation*}
where $\Phi=({\Phi^{\alpha}}_{\beta})$ is a local endomorphism of $TM$ given by \cite{ChToWe16a}
\begin{eqnarray*}
{\Phi^{\alpha}}_{\beta}&=& g^{\alpha\gamma}\nabla^2_{\gamma\beta}\varphi-g^{\alpha\gamma}B_{\gamma\beta},\\
B_{\alpha\beta}&=& \delta_{\alpha\beta}-{V^\alpha}_1{V^\beta}_1.
\end{eqnarray*}
Here $\{{V^\alpha}_1\}_{\alpha=1}^{2n}$ are the components of $V_1$ at $p$. Then
$p$ is still local maximum point of $\hat{Q}$. By the definition of $\Phi$, at $p$, $V_{1}, V_{2}, \cdots, V_{2n}$ are eigenvectors of $\Phi$ corresponding to eigenvalues $\lambda_{1}(\Phi)>\lambda_{2}(\Phi)\geq\cdots\geq\lambda_{2n}(\Phi)$. For convenience, in the following argument, we use $\lambda_{\alpha}$ and $\vp_{V_{\alpha}V_{\beta}}$ to denote $\lambda_{\alpha}(\Phi)$ and $\nabla^{2}\vp(V_{\alpha},V_{\beta})$ respectively.

\begin{lemma}\label{Preliminary Computations}
There exists a uniform $C>0$ such that
if $\lambda_{1}\geq C\displaystyle\sup_{M}|\partial\vp|_{g}+C\|F\|_{C^{2}}$ at $p$, then we have
\begin{eqnarray}\label{Gradient Computation}
L(|\partial\vp|_{g}^{2})&\geq& \frac{1}{2}\sum_{k}\tilde{g}^{i\overline{i}}(|e_{i}e_{k}(\vp)|^{2}+|e_{i}\overline{e}_{k}(\vp)|^{2})
-C\sum_{i}\tilde{g}^{i\overline{i}}\nonumber\\
&&-\frac{1}{\ve}\left(3\sup_{M}|\partial\vp|_{g}^{2}+\|F\|_{C^{1}}^{2}\right)
\end{eqnarray}
and
\begin{eqnarray}\label{Lambda 1 computation}
L(\lambda_{1}) &\geq& 2\sum_{\alpha>1}\tilde{g}^{i\overline{i}}\frac{|e_{i}(\vp_{V_{\alpha}V_{1}})|^{2}}{\lambda_{1}-\lambda_{\alpha}}
+\tilde{g}^{p\overline{p}}\tilde{g}^{q\overline{q}}|V_{1}(\tilde{g}_{p\overline{q}})|^{2}-2\tilde{g}^{i\overline{i}}[V_{1},e_{i}]V_{1}\overline{e}_{i}(\vp)\nonumber\\
&&-2\tilde{g}^{i\overline{i}}[V_{1},\overline{e}_{i}]V_{1}e_{i}(\vp)-C\lambda_{1}\sum_{i}\tilde{g}^{i\overline{i}}+\frac{1}{2\ve}\lambda_{1}.
\end{eqnarray}
where $L=\tilde g^{i\bar j}(e_i\bar e_j-[e_i, \bar e_j]^{(1,0)})$ is the operator defined in \cite[p.12]{ChToWe16a}.
\end{lemma}

\begin{proof}
For (\ref{Gradient Computation}), by (4.8) in \cite{ChToWe16a} (as before, we replace $\ve$ and $F$ in (4.13) by $\delta$ and $\frac{1}{\ve}(\vp-F)$ to avoid confusion of notations), we have
\begin{eqnarray*}
L(|\partial\vp|_{g}^{2}) & \geq & (1-\delta)\sum_{k}\tilde{g}^{i\overline{i}}(|e_{i}e_{k}(\vp)|^{2}+|e_{i}\overline{e}_{k}(\vp)|^{2})
-C\delta^{-1}|\partial\vp|_{g}^{2}\sum_{i}\tilde{g}^{i\overline{i}}\\
&&+2\mathrm{Re}\left(\sum_{i}\frac{1}{\ve}(\vp_{i}-F_{i})\vp_{\overline{i}}\right)
\end{eqnarray*}
at $p$. Now we take $\delta=\frac{1}{2}$. By Proposition \ref{Gradient estimate}, we obtain
\begin{eqnarray}\label{Preliminary Computations equation 1}
L(|\partial\vp|_{g}^{2}) & \geq & \frac{1}{2}\sum_{k}\tilde{g}^{i\overline{i}}(|e_{i}e_{k}(\vp)|^{2}+|e_{i}\overline{e}_{k}(\vp)|^{2})
-C\sum_{i}\tilde{g}^{i\overline{i}}\nonumber\\
&&+2\mathrm{Re}\left(\sum_{i}\frac{1}{\ve}(\vp_{i}-F_{i})\vp_{\overline{i}}\right).
\end{eqnarray}
By Cauchy inequality, it is clear that
\begin{equation}\label{Preliminary Computations equation 2}
2\mathrm{Re}\left(\sum_{i}\frac{1}{\ve}(\vp_{i}-F_{i})\vp_{\overline{i}}\right)
\geq -\frac{1}{\ve}\left(3\sup_{M}|\partial\vp|_{g}^{2}+\|F\|_{C^{1}}^{2}\right).
\end{equation}
Combining (\ref{Preliminary Computations equation 1}) and (\ref{Preliminary Computations equation 2}), we obtain (\ref{Gradient Computation}).

For (\ref{Lambda 1 computation}), by (5.11) and (5.12) in \cite{ChToWe16a}, we have
\begin{eqnarray}\label{Preliminary Computations equation 3}
L(\lambda_{1}) & \geq & 2\sum_{\alpha>1}\tilde{g}^{i\overline{i}}\frac{|e_{i}(\vp_{V_{\alpha}V_{1}})|^{2}}{\lambda_{1}-\lambda_{\alpha}}
+\tilde{g}^{i\overline{i}}V_{1}V_{1}(\tilde{g}_{i\overline{i}})-2\tilde{g}^{i\overline{i}}[V_{1},e_{i}]V_{1}\overline{e}_{i}(\vp)\nonumber\\
&&-2\tilde{g}^{i\overline{i}}[V_{1},\overline{e}_{i}]V_{1}e_{i}(\vp)-C\lambda_{1}\sum_{i}\tilde{g}^{i\overline{i}}
\end{eqnarray}
at $p$. In the local frame $\{e_{i}\}_{i=1}^{n}$, the complex Monge-Amp\`{e}re equation (\ref{CMAE}) can be written as
\begin{equation*}
\log\det\tilde{g}=\frac{1}{\ve}(\vp-F).
\end{equation*}
Differentiating the equation twice with $V_{1}$, we obtain
\begin{equation}\label{Preliminary Computations equation 4}
\tilde{g}^{i\overline{i}}V_{1}V_{1}(\tilde{g}_{i\overline{i}})=\tilde{g}^{p\overline{p}}\tilde{g}^{q\overline{q}}|V_{1}(g_{p\overline{q}})|^{2}
+V_{1}V_{1}\left(\frac{1}{\ve}(\vp-F)\right).
\end{equation}
Assume $\lambda_{1}\geq C\displaystyle\sup_{M}|\partial\vp|_{g}+C\|F\|_{C^{2}}$ at $p$.
When $C$ is sufficiently large,
\begin{eqnarray}\label{Preliminary Computations equation 5}
V_{1}V_{1}\left(\frac{1}{\ve}(\vp-F)\right)
& = & \frac{1}{\ve}\left(\lambda_{1}+(\nabla_{V_{1}}V_{1})\vp-V_{1}V_{1}(F)\right)\nonumber\\
& \geq & \frac{1}{2\ve}\lambda_{1}.
\end{eqnarray}
Then (\ref{Lambda 1 computation}) follows from (\ref{Preliminary Computations equation 3}), (\ref{Preliminary Computations equation 4}) and (\ref{Preliminary Computations equation 5}).
\end{proof}

\begin{lemma}\label{Analogue of Lemma 5.4}There exists a uniform $C>0$ such that
if $\lambda_{1}\geq C\displaystyle\sup_{M}|\partial\vp|_{g}+C\|F\|_{C^{2}}$ at $p$,
then for any $\delta\in(0,\frac{1}{2}]$, we have
\begin{eqnarray*}
0 &\geq & (2-\delta)\sum_{\alpha>1}\tilde{g}^{i\overline{i}}\frac{|e_{i}(\vp_{V_{\alpha}V_{1}})|^{2}}{\lambda_{1}(\lambda_{1}-\lambda_{\alpha})}
+\frac{\tilde{g}^{p\overline{p}}\tilde{g}^{q\overline{q}}|V_{1}(\tilde{g}_{p\overline{q}})|^{2}}{\lambda_{1}}
-(1+\delta)\frac{\tilde{g}^{i\overline{i}}|e_{i}(\vp_{V_{1}V_{1}})|^{2}}{\lambda_{1}^{2}}\\[4pt]
&& + \frac{h_{D}'}{2}\sum_{k}\tilde{g}^{i\overline{i}}(|e_{i}e_{k}(\vp)|^{2}+|e_{i}\overline{e}_{k}(\vp)|^{2})
+h_{D}''\tilde{g}^{i\overline{i}}|\partial_{i}|\partial\vp|_{g}^{2}|^{2}\\
&& + (Ae^{-A\vp}-\frac{C}{\delta})\sum_{i}\tilde{g}^{i\overline{i}}+A^{2}e^{-A\vp}\tilde{g}^{i\overline{i}}|e_{i}(\vp)|^{2}-Ane^{-A\vp}.
\end{eqnarray*}
\end{lemma}

\begin{proof}
First, by direct calculations, at $p$, we have
\begin{eqnarray}\label{Analogue of Lemma 5.4 equation 1}
L(\hat{Q}) &=& \frac{L(\lambda_{1})}{\lambda_{1}}-\frac{\tilde{g}^{i\overline{i}}|e_{i}(\lambda_{1})|^{2}}{\lambda_{1}^{2}}
+h_{D}'L(|\partial\vp|_{g}^{2})+h_{D}''\tilde{g}^{i\overline{i}}|e_{i}|\partial\vp|_{g}^{2}|^{2}\nonumber\\
&& -Ae^{-A\vp}L(\vp)+A^{2}e^{-A\vp}\tilde{g}^{i\overline{i}}|e_{i}(\vp)|^{2}.
\end{eqnarray}
By the proof of Lemma 5.4 in \cite{ChToWe16a}, for any $\delta\in(0,\frac{1}{2}]$, we get
\begin{eqnarray}\label{Analogue of Lemma 5.4 equation 2}
&& 2\frac{\tilde{g}^{i\overline{i}}[V_{1},e_{i}]V_{1}\overline{e}_{i}(\vp)
+\tilde{g}^{i\overline{i}}[V_{1},\overline{e}_{i}]V_{1}e_{i}(\vp)}{\lambda_{1}}\nonumber\\
& \leq & \delta\frac{\tilde{g}^{i\overline{i}}|e_{i}(\vp_{V_{1}V_{1}})|^{2}}{\lambda_{1}^{2}}
+\delta\sum_{\alpha>1}\tilde{g}^{i\overline{i}}\frac{|e_{i}(\vp_{V_{\alpha}V_{1}})|^{2}}{\lambda_{1}(\lambda_{1}-\lambda_{\alpha})}
+\frac{C}{\delta}\sum_{i}\tilde{g}^{i\overline{i}}.
\end{eqnarray}
For the first term of (\ref{Analogue of Lemma 5.4 equation 1}), by (\ref{Lambda 1 computation}) and (\ref{Analogue of Lemma 5.4 equation 2}),
\begin{eqnarray}\label{Analogue of Lemma 5.4 equation 3}
\frac{L(\lambda_{1})}{\lambda_{1}} & \geq &
(2-\delta)\sum_{\alpha>1}\tilde{g}^{i\overline{i}}\frac{|e_{i}(\vp_{V_{\alpha}V_{1}})|^{2}}{\lambda_{1}(\lambda_{1}-\lambda_{\alpha})}
+\frac{\tilde{g}^{p\overline{p}}\tilde{g}^{q\overline{q}}|V_{1}(\tilde{g}_{p\overline{q}})|^{2}}{\lambda_{1}}\nonumber\\
&& -\delta\frac{\tilde{g}^{i\overline{i}}|e_{i}(\vp_{V_{1}V_{1}})|^{2}}{\lambda_{1}^{2}} -\frac{C}{\delta}\sum_{i}\tilde{g}^{i\overline{i}}+\frac{1}{2\ve}.
\end{eqnarray}
For the second term of (\ref{Analogue of Lemma 5.4 equation 1}), by Lemma 5.2 in \cite{ChToWe16a}, we obtain
\begin{equation}\label{Analogue of Lemma 5.4 equation 4}
-\frac{\tilde{g}^{i\overline{i}}|e_{i}(\lambda_{1})|^{2}}{\lambda_{1}^{2}}
=-\frac{\tilde{g}^{i\overline{i}}|e_{i}(\vp_{V_{1}V_{1}})|^{2}}{\lambda_{1}^{2}}.
\end{equation}
For the third term of (\ref{Analogue of Lemma 5.4 equation 1}), by (\ref{Properties of h D}) and (\ref{Gradient Computation}), we have
\begin{eqnarray}\label{Analogue of Lemma 5.4 equation 5}
h_{D}'L(|\partial\vp|_{g}^{2}) &\geq&
\frac{h_{D}'}{2}\sum_{k}\tilde{g}^{i\overline{i}}(|e_{i}e_{k}(\vp)|^{2}+|e_{i}\overline{e}_{k}(\vp)|^{2})
-\frac{C}{2D}\sum_{i}\tilde{g}^{i\overline{i}}\nonumber\\
&& -\frac{h_{D}'}{\ve}(3\sup_{M}|\partial\vp|_{g}^{2}+\|F\|_{C^{1}}^{2}).
\end{eqnarray}
Now we choose $D=3\displaystyle\sup_{M}|\partial\vp|_{g}^{2}+\|F\|_{C^{1}}^{2}$. Then by (\ref{Properties of h D}),
\begin{equation}\label{Analogue of Lemma 5.4 equation 6}
\frac{1}{2\ve}-\frac{h_{D}'}{\ve}(3\sup_{M}|\partial\vp|_{g}^{2}+\|F\|_{C^{1}}^{2})\geq 0.
\end{equation}
For the fifth term of (\ref{Analogue of Lemma 5.4 equation 1}), we have
\begin{equation}\label{Analogue of Lemma 5.4 equation 7}
-Ae^{-A\vp}L(\vp)=Ae^{-A\vp}\sum_{i}\tilde{g}^{i\overline{i}}-Ane^{-A\vp}.
\end{equation}
Therefore, combining $L(\hat{Q})(p)\leq 0$, (\ref{Analogue of Lemma 5.4 equation 1}), (\ref{Analogue of Lemma 5.4 equation 2}), (\ref{Analogue of Lemma 5.4 equation 3}), (\ref{Analogue of Lemma 5.4 equation 4}), (\ref{Analogue of Lemma 5.4 equation 5}), (\ref{Analogue of Lemma 5.4 equation 6}) and (\ref{Analogue of Lemma 5.4 equation 7}), we complete the proof.
\end{proof}

Lemma \ref{Analogue of Lemma 5.4} is just the analogue of \cite[Lemma 5.4]{ChToWe16a}.
Finally, by Proposition \ref{Zero order estimate}, \ref{Gradient estimate}, Lemma \ref{Analogue of Lemma 5.4} and the similar argument of \cite[Proposition 5.1]{ChToWe16a}, we obtain the uniform upper bound of $\lambda_{1}$ at $p$, which completes the proof of Proposition \ref{Second order estimate}.  There are two problems that need to be explained.

One is that $h_{D}$ in this paper is different from $h$ in \cite{ChToWe16a} (the definition of $h$ corresponds to the definition of $h_{D}$ when $D=1$, i.e., $h=h_{1}$). However, the reader can verify, this minor difference does not influence the argument.

The other is that in the proof of \cite[Proposition 5.1]{ChToWe16a} the lower bound of $\frac{\tilde{\omega}^{n}}{\omega^{n}}$ is used, while in our case Proposition \ref{Zero order estimate} only guarantees the upper bound of $\displaystyle\sup_{M}\frac{\tilde{\omega}^{n}}{\omega^{n}}$.
However, we point out that $\inf_{M}\frac{\tilde{\omega}^{n}}{\omega^{n}}$ is not needed in the proof of \cite[Proposition 5.1]{ChToWe16a}. In fact, the proof of \cite[Proposition 5.1]{ChToWe16a} is split up into different cases.
In Case 1(a), we have (see \cite[p.23]{ChToWe16a})
\begin{equation*}
0\geq \sum_{k}\tilde{g}^{i\overline{i}}(|e_{i}e_{k}(\vp)|^{2}+|e_{i}\overline{e}_{k}(\vp)|^{2})-C\sum_{i}\tilde{g}^{i\overline{i}}.
\end{equation*}
By $\tilde{g}_{1\overline{1}}\geq\tilde{g}_{2\overline{2}}\geq\cdots\geq\tilde{g}_{n\overline{n}}$, we obtain
\begin{equation*}
0\geq \sum_{i,k}\tilde{g}^{1\overline{1}}(|e_{i}e_{k}(\vp)|^{2}+|e_{i}\overline{e}_{k}(\vp)|^{2})-Cn\tilde{g}^{n\overline{n}}.
\end{equation*}
Combining this and assumption $\tilde{g}_{1\overline{1}}\leq A^{3}e^{-2A\vp}\tilde{g}_{n\overline{n}}$, it is clear that
\begin{equation*}
0\geq \sum_{i,k}(|e_{i}e_{k}(\vp)|^{2}+|e_{i}\overline{e}_{k}(\vp)|^{2})-C,
\end{equation*}
which is enough for the following proof.
In Case 1(b) and Case 2, we have $\displaystyle\sum_{i}\tilde{g}^{i\overline{i}}\leq C$, where $C$ is independent of $\displaystyle\inf_{M}\frac{\tilde{\omega}^{n}}{\omega^{n}}$ (see \cite[p.12 and p.33]{ChToWe16a}). Combining this and $\displaystyle\prod_{i}\tilde{g}_{i\overline{i}}=\frac{\tilde{\omega}^{n}}{\omega^{n}}$, for each $i$, we obtain
\begin{equation*}
\tilde{g}_{i\overline{i}}\leq C\sup_{M}\frac{\tilde{\omega}^{n}}{\omega^{n}},
\end{equation*}
which implies $\tilde{g}^{i\overline{i}}\geq C^{-1}$ (independent of $\displaystyle\inf_{M}\frac{\tilde{\omega}^{n}}{\omega^{n}}$). And this is also enough for the following proof.
Therefore, the estimate in \cite[Proposition 5.1]{ChToWe16a} is independent of $\displaystyle\inf_{M}\frac{\tilde{\omega}^{n}}{\omega^{n}}$.
\end{proof}

\section{Regularity of envelope}

In this section, we prove the $C^{1,1}$-regularity of envelopes.
First, we recall the regularization theorem of plurisubharmonic functions on Hermitian manifolds.

\begin{theorem}\label{Regularization Theorem1} \cite{BK, Dem92, Dem94, KN}
Let $(M,\omega)$ be a compact Hermitian manifold. For any $\vp\in PSH(M,\omega)$, there exists a sequence $\vp_{i}\in PSH(M,\omega)\cap C^{\infty}(M)$ such that $\vp_{i}$ converges decreasingly to $\vp$.
\end{theorem}


Next lemma can be regarded as a special case of Theorem \ref{C 1,1 regularity of envelope}.

\begin{lemma}\label{C 1,1 regularity of envelope, smooth case}
Let $(M,\omega)$ be a compact Hermitian manifold. For any $f\in C^{\infty}(M)$, we have $\vp_{f}\in C^{1,1}(M)$ and
\begin{equation*}
\|\vp_{f}\|_{C^{1,1}}\leq C,
\end{equation*}
where $C$ is a constant depending only on $\|f\|_{C^{2}}$ and $(M,\omega)$.
\end{lemma}

\begin{proof}
We consider the complex Monge-Amp\`{e}re equation
\begin{equation*}
(\omega+\ddbar\vp)^{n}=e^{\frac{1}{\ve}(\vp-f)}\omega^{n}.
\end{equation*}
We use $\vp_{\ve}$ to denote its unique solution, i.e.,
\begin{equation}\label{C 1,1 regularity of envelope, smooth case equation 1}
(\omega+\ddbar\vp_{\ve})^{n}=e^{\frac{1}{\ve}(\vp_{\ve}-f)}\omega^{n}.
\end{equation}
Next, for any $u\in PSH(M,\omega)\cap C^{\infty}(M)$ such that $u\leq f$, we define
\begin{equation*}
u_{\ve}=(1-\ve)u+\ve(\log\ve^{n}+\min_{M}f).
\end{equation*}
By direct calculation, we have
\begin{equation}\label{C 1,1 regularity of envelope, smooth case equation 2}
(\omega+\ddbar u_{\ve})^{n}\geq \ve^{n}\omega^{n}\geq e^{\frac{1}{\ve}(u_{\ve}-f)}\omega^{n}.
\end{equation}
Combining (\ref{C 1,1 regularity of envelope, smooth case equation 1}), (\ref{C 1,1 regularity of envelope, smooth case equation 2}) and maximum principle, we obtain $u_{\ve}\leq\vp_{\ve}$, which implies
\begin{equation*}
(1-\ve)u+\ve(\log\ve^{n}-\|f\|_{L^{\infty}})\leq\vp_{\ve}.
\end{equation*}
Theorem \ref{Regularization Theorem1} implies
\begin{equation*}
\vp_{f}(x) = \sup \{\vp(x)~|~\vp\in PSH(M,\omega)\cap C^{\infty}(M) \text{~and~} \vp\leq f\}.
\end{equation*}
Since $u$ is arbitrary, by $u\leq f$, it is clear that
\begin{equation*}
(1-\ve)\vp_{f}+\ve(\log\ve^{n}-\|f\|_{L^{\infty}})\leq\vp_{\ve},
\end{equation*}
which implies
\begin{equation}\label{C 1,1 regularity of envelope, smooth case equation 3}
\vp_{f}+\ve(\log\ve^{n}-2\|f\|_{L^{\infty}})\leq\vp_{\ve},
\end{equation}
where we used $\vp_{f}\leq f$. By Proposition \ref{Zero order estimate} and the definition of $\vp_{\ve}$, we obtain
\begin{equation}\label{C 1,1 regularity of envelope, smooth case equation 6}
\vp_{\ve}-C_{0}\ve\leq f \text{~and~} \vp_{\ve}-C_{0}\ve\in PSH(M,\omega).
\end{equation}
Combining (\ref{C 1,1 regularity of envelope, smooth case equation 3}) and (\ref{C 1,1 regularity of envelope, smooth case equation 6}), it is clear that
\begin{equation}\label{C 1,1 regularity of envelope, smooth case equation 4}
\lim_{\ve\rightarrow 0} \|\vp_{\ve}-\vp_{f}\|_{L^{\infty}}=0.
\end{equation}
On the other hand, by Proposition \ref{Zero order estimate},  \ref{Gradient estimate} and \ref{Second order estimate}, we have
\begin{equation}\label{C 1,1 regularity of envelope, smooth case equation 5}
\|\vp_{\ve}\|_{C^{2}}\leq C.
\end{equation}
Combining (\ref{C 1,1 regularity of envelope, smooth case equation 4}) and (\ref{C 1,1 regularity of envelope, smooth case equation 5}), we have $\vp_\ve$ converges in $C^{1,1}$ to $\varphi_f$.
\end{proof}

Now we are in a position to prove Theorem \ref{C 1,1 regularity of envelope}.

\begin{proof}[Proof of Theorem \ref{C 1,1 regularity of envelope}]
Since $f\in C^{1,1}(M)$, by smooth approximation, there exists a sequence of smooth function $f_{i}$ on $M$ such that
\begin{equation}\label{C 1,1 regularity of envelope equation 1}
\lim_{i\rightarrow\infty}\|f_{i}-f\|_{L^{\infty}}=0 \text{~and~} \|f_{i}\|_{C^{2}}\leq C_{0},
\end{equation}
where $C_{0}$ is a constant depending only on $\|f\|_{C^{1,1}}$ and $(M,\omega)$. On the other hand, for any $u\in PSH(M,\omega)$ with $u\leq f$, we have
\begin{equation*}
u-\|f_{i}-f\|_{L^{\infty}}\leq f_{i}.
\end{equation*}
By the definition of $\vp_{f_{i}}$, it is clear that
\begin{equation*}
u-\|f_{i}-f\|_{L^{\infty}}\leq \vp_{f_{i}}.
\end{equation*}
Since $u$ is arbitrary, we obtain
\begin{equation*}
\vp_{f}-\|f_{i}-f\|_{L^{\infty}}\leq \vp_{f_{i}}.
\end{equation*}
Similarly, we get
\begin{equation*}
\vp_{f_{i}}-\|f-f_{i}\|_{L^{\infty}}\leq \vp_{f}.
\end{equation*}
It then follows that
\begin{equation}\label{C 1,1 regularity of envelope equation 2}
\|\vp_{f_{i}}-\vp_{f}\|_{L^{\infty}} \leq \|f_{i}-f\|_{L^{\infty}}.
\end{equation}
Combining (\ref{C 1,1 regularity of envelope equation 1}) and (\ref{C 1,1 regularity of envelope equation 2}), we complete the proof.
\end{proof}

\section{Examples}

In this section, we give some examples to explain the regularity result in Theorem \ref{C 1,1 regularity of envelope} is optimal.

\subsection{Example of complex dimension one}
In this subsection, we construct a smooth function $f$ on the complex projective space $(\mathbb{CP}^{1},\omega_{FS})$,  such that $\vp_{f}\notin C^{2}(\mathbb{CP}^{1})$, where $\omega_{FS}$ is the Fubini-Study metric.

First, we define a function $h(t)$  on $[0,2]$ by
\begin{equation*}
h(t)=\begin{cases}
\left(\frac{1}{\sqrt{3}-1}-1\right)^{2}+\log(\sqrt{3}-1), & t\in [0,\sqrt{3}-1],\\[6pt]
\left((\frac{1}{t}-1)_{+}\right)^{2}+\log t, & t\in [\sqrt{3}-1,2],
\end{cases}
\end{equation*}
where $(\frac{1}{t}-1)_{+}=\max\{\frac{1}{t}-1,0\}$. It is clear that $h$ is a convex function on $[0,1]$. Let $\tilde{h}$ be a smooth function on $[0,2]$ such that
\begin{equation*}
\text{$\tilde{h}(t)=\left(\frac{1}{t}-1\right)^{2}+\log t$ in $[\frac{4}{5},2]$ and $\tilde{h}(t)\geq h(t)$ in $[0,2]$}.
\end{equation*}
Denote by $[z_0, z_1]$ the homogeneous coordinates on $\mathbb{CP}^{1}$.
Let
\begin{equation*}
U=\{[1,z_{1}]~|~|z_{1}|^{2}\leq\frac{5}{4}\} \text{~and~} V=\{[z_{0},1]~|~|z_{0}|^{2}\leq2\}
\end{equation*} be two subsets of $\mathbb{CP}^{1}$ such that $\mathbb{CP}^{1}=U \cup V$.
We define a function $f$ on $\mathbb{CP}^{1}$ by
\begin{equation*}
f=\left\{ \begin{array}{ll}
\ \text{$(|z_{1}|^{2}-1)^{2}-\log(1+|z_{1}|^{2})$ in $U$,} \\[6pt]
\ \text{$\tilde{h}(|z_{0}|^{2})-\log(1+|z_{0}|^{2})$ in $V$.}
\end{array}\right.
\end{equation*}
Since $\tilde{h}\in C^{\infty}([0,2])$, we obtain that $f\in C^{\infty}(\mathbb{CP}^{1})$.
Then we prove
\begin{proposition}\label{counter-example}
$\vp_{f}\in C^{1,1}(\mathbb{CP}^{1})\setminus C^{2}(\mathbb{CP}^{1})$.
\end{proposition}

\begin{proof}
Define
\begin{equation*}
\vp=\left\{ \begin{array}{ll}
\ \text{$\left((|z_{1}|^{2}-1)_{+}\right)^{2}-\log(1+|z_{1}|^{2})$ in $U$,} \\[6pt]
\ \text{$h(|z_{0}|^{2})-\log(1+|z_{0}|^{2})$ in $V$.}
\end{array}\right.
\end{equation*}
It is clear that $\vp\in C^{1,1}(U)\setminus C^{2}(U)$. Since
$\tilde{h}\geq h$, we have  $\vp\leq f$.

Next, we verify that $\vp\in PSH(\mathbb{CP}^{1},\omega_{FS})$. On $U$, we compute
\begin{eqnarray*}
&&\omega_{FS}+\ddbar\vp \\
& = &\ddbar \left[\log(1+|z_{1}|^{2})+\left((|z_{1}|^{2}-1)_{+}\right)^{2}-\log(1+|z_{1}|^{2})\right]\\
& \geq & 0.
\end{eqnarray*}
Similarly, on $V_{0}=\{[z_{0},1]~|~|z_{0}|^{2}\leq 1\}$, we have
\begin{eqnarray*}
&&\omega_{FS}+\ddbar\vp \\ & =& \ddbar \left[\log(1+|z_{1}|^{2})+h(|z_{0}|^{2})-\log(1+|z_{1}|^{2})\right]\\
& =& \ddbar \left(h(|z_{0}|^{2})\right)\\
& \geq & 0,
\end{eqnarray*}
where we used the fact that $h$ is a convex function on $[0,1]$. Since $\mathbb{CP}^{1}=U \cup V_{0}$, we obtain $\vp\in PSH(\mathbb{CP}^{1},\omega_{FS})$.

Now we show $\varphi_f$ is not $C^{2}$.
For convenience, we denote
$$U_0=\{[1,z_{1}]~|~|z_{1}|^{2}\leq 1\}.$$ Then for any $u\in PSH(\mathbb{CP}^{1},\omega_{FS})$ such that $u\leq f$, we have
\begin{equation*}
\omega_{FS}+\ddbar u=\ddbar\left[\log(1+|z_{1}|^{2})+u)\right]\geq 0 \text{~~in $U_0$}
\end{equation*}
and
\begin{equation*}
\log(1+|z_{1}|^{2})+u\leq \log(1+|z_{1}|^{2})+f=0 \text{~on $\partial U_0$}.
\end{equation*}
By maximum principle, it is clear that
\begin{equation}\label{One dimensional example equation 1}
u\leq -\log(1+|z_{1}|^{2})=\vp \text{~in $U_0$}.
\end{equation}
Since $\vp=f$ on $U\setminus U_0$ and $u\leq f$, we have
\begin{equation}\label{One dimensional example equation 2}
u\leq f=\vp \text{~in $U\setminus U_0$}.
\end{equation}
Combining (\ref{One dimensional example equation 1}) and (\ref{One dimensional example equation 2}), we have $u\leq \vp$ on $U$,
which implies $\vp_{f}=\vp$ on $U$.
The proposition is proved.
\end{proof}

\subsection{Examples of higher dimensions}
In this subsection, we give more examples on compact Hermitian manifolds of higher dimensions. First, we have the following lemma.

\begin{lemma}\label{Higher dimensional examples lemma}
Let $(M,\omega_{M})$ and $(N,\omega_{N})$ be Hermitian manifolds and let $\pi:M\times N\rightarrow M$ be the projection map. For any $f\in C^{1,1}(M)$, we have
\begin{equation*}
\pi^{*}\vp_{f}=\vp_{\pi^{*}f},
\end{equation*}
where $\pi^{*}$ is the pullback map.
\end{lemma}

\begin{proof}
First, since $\vp_{f}\in PSH(M,\omega_{M})$ and $\vp_{f}\leq f$, we obtain $\pi^{*}\vp_{f}\in PSH(M\times N,\omega_{M}+\omega_{N})$ and $\pi^{*}\vp_{f}\leq\pi^{*}f$, which implies
\begin{equation}\label{Higher dimensional examples lemma equation 1}
\pi^{*}\vp_{f}\leq\vp_{\pi^{*}f}.
\end{equation}
Next, for any $(p.q)\in M\times N$ and $u\in PSH(M\times N,\omega_{M}+\omega_{N})$ with $u\leq\pi^{*}f$, we have $u(\cdot,q)\in PSH(M,\omega_{M})$ and $u(\cdot,q)\leq f$ on $M$. It then follows that $u(p,q)\leq\vp_{f}(p)$, i.e.,
\begin{equation*}
u(p,q)\leq\pi^{*}\vp_{f}(p,q).
\end{equation*}
Since $(p,q)$ and $u$ are arbitrary, by the definition of $\vp_{\pi^{*}f}$, we have
\begin{equation}\label{Higher dimensional examples lemma equation 2}
\vp_{\pi^{*}f}\leq\pi^{*}\vp_{f}.
\end{equation}
Combining (\ref{Higher dimensional examples lemma equation 1}) and (\ref{Higher dimensional examples lemma equation 2}), we complete the proof.
\end{proof}

Now, let $(M,\omega)$ be a compact Hermitian manifold and let $\pi:\mathbb{CP}^{1}\times M\rightarrow \mathbb{CP}^{1}$ be the projection map. Then $\pi^{*}f$ is a smooth function on $\mathbb{CP}^{1}\times M$, where $f$ is defined in subsection 4.1. However, by Proposition \ref{counter-example} and Lemma \ref{Higher dimensional examples lemma}, $\vp_{\pi^{*}f}=\pi^{*}\vp_{f}$ is not $C^{2}$.

\end{document}